\documentclass[12pt]{amsart}
\usepackage{amsmath,amssymb,amsfonts}
\usepackage[all]{xy}
\usepackage{hyperref}
\usepackage{a4wide}

\newtheorem{thm}{Theorem}[section]
\newtheorem{lem}[thm]{Lemma}
\newtheorem{cor}[thm]{Corollary}
\newtheorem{prop}[thm]{Proposition}

\newtheorem{conj}[thm]{Conjecture}

\theoremstyle{definition}
\newtheorem{rmk}[thm]{Remark}

\newcommand{\cd}{{\rm cd}}

\newcommand{\im}{{\rm im}}

\newcommand{\Gal}{{\rm Gal}}

\newcommand{\sE}{{\mathcal E}}

\newcommand{\G}{{\mathbb G}}

\begin{document}
\title[]{On $p$-embedding problems in characteristic $p$} 
 \author{Lior Bary-Soroker}
\address{
Institut f\"ur Experimentelle Mathematik,
Universit\"at Duisburg-Essen,
Ellernstrasse 29,
D-45326 Essen,
Germany}
\email{lior.bary-soroker@uni-due.de}%

 \author{ Nguy\^e\~n Duy T\^an }
 \address{ Universit\"at Duisburg-Essen, FB6, Mathematik, 45117 Essen, Germany \\
 and Institute of Mathematics, 18 Hoang Quoc Viet, 10307, Hanoi - Vietnam.}
\email{duy-tan.nguyen@uni-due.de}
\thanks{The second author is partially supported by NAFOSTED, SFB/TR45 and the ERC/Advanced Grant 226257.}

\date{August 11, 2010} 

\begin{abstract} Let $K$ be a valued field of characteristic $p>0$ with non-$p$-divisible value group. We show that every finite embedding problem for $K$ whose kernel is a $p$-group is  properly solvable. 

\end{abstract}
\maketitle

\section{Introduction}
In the proof that every finite solvable group occurs as a Galois group over the rationals, Shafarevich studies the solvability of embedding problems with nilpotent kernel and solvable cokernel. To study the absolute Galois group $\Gal(K)$ of a field $K$ via embedding problems became a trend in recent papers, e.g.\ \cite{BHH,Ha3,HS, Pa1,Po}. See also the upcoming book \cite{Ja} and references therein. 
 
In this work we consider a field $K$ of characteristic $p>0$ and the finite embedding problems for $K$ whose kernels are $p$-groups which we call \textbf{finite $p$-embedding problems}. An obvious necessary condition to have a \emph{proper} solution is to have a \emph{weak} solution (see Section~\ref{sec:3} for definitions). This latter condition is automatically satisfied in our case, since $\cd_p(\Gal(K))\leq 1$, for a field of characteristic $p>0$. We obtain a mild sufficient condition on $K$ to have a proper solution of any finite $p$-embedding problem.

\begin{thm}
\label{thm:main}
Let $K$ be a field of characteristic $p$ admitting a non-$p$-divisible valuation. Then every finite $p$-embedding problem for $K$ is solvable.
\end{thm}

Some examples of fields satisfying this condition are the following. If $R$ is a Noetherian domain or Krull domain of characteristic $p>0$, then its  fraction field $K$ satisfies the hypothesis of Theorem~\ref{thm:main}. If $R$ is an arbitrary domain of characteristic $p>0$, then the fraction fields of the ring $R[x_1, \ldots, x_n]$ of polynomials and  of the ring of formal Taylor series $R[[x_1, \ldots, x_n]]$ satisfy the hypothesis of Theorem~\ref{thm:main}, for any $n\geq 1$.

The proof of Theorem~\ref{thm:main} is based on the following cohomological criterion of Harbater. A profinite group $\Pi$ is called {\bf strongly $p$-dominating} if $H^1(\Pi, P)$ is infinite for every nontrivial finite elementary $p$-group $P$ on which $\Pi$ acts\footnote{All actions, homomorphism, etc., in this work are assumed to be continuous.}.

\begin{thm}[{\cite[Theorem 1b]{Ha2}}]
\label{thm:Harbater}
 Let $\Pi$ be a profinite group. Assume that $\Pi$ is strongly $p$-dominating and that $\cd_p(\Pi)\leq 1$. Then every finite $p$-embedding problem for $\Pi$ is properly solvable.
\end{thm}

Harbater's motivation for Theorem~\ref{thm:Harbater} is to show that every finite $p$-embedding problem for the \'etale fundamental group $\Pi:=\pi_1(X)$ of an affine variety $X$ over an arbitrary field $K$ of characteristic $p>0$ has a  proper solution \cite{Ha1}. 

We show that the converse of Theorem~\ref{thm:Harbater} also holds true, see Theorem~\ref{thm:converse}. Moreover,  to get the assertion of Theorem~\ref{thm:Harbater}, one may suspect that the infinitude of  $H^1(\Pi, \mathbb{Z}/p\mathbb{Z})$ suffices, where $\Pi$ acts trivially on $\mathbb{Z}/p\mathbb{Z}$. This is true if both the kernel and cokernel are $p$-groups, but in general it fails, see \cite{Ha2}. 

By Theorem~\ref{thm:Harbater}, to prove Theorem~\ref{thm:main} it suffices to show that $\Gal(K)$ is strongly $p$-dominating. This is carried out by using that 
for every nontrivial finite elementary $p$-group $P$ on which $\Pi$ acts we have $H^1(K,P) = K/f(K)$, for some additive polynomial $f$ (Lemma~\ref{lem:additivepol}). Then using the non-$p$-divisible valuation of $K$ we construct infinitely many $a\in K$ that are distinct modulo $f$.  

We conclude the introduction with an example. Let $K_0$ be a field of characteristic $p>0$ and $K=K_0((x))$ the field of formal Laurent series. Then by Theorem~\ref{thm:main} every finite $p$-embedding problem is properly solvable. When $K_0$ is algebraically closed, Harbater proves this in \cite[Example 5]{Ha2} using a similar method. However, when $K_0$ is arbitrary Harbater invokes a theorem of Katz-Gabber in order to complete his proof (see \cite[Proposition 6]{Ha2}).

\noindent {\bf Acknowledegements:} We would like to give our sincere thanks to H\'el\`ene Esnault for her support and constant encouragement. 

The first author is an Alexander von Humboldt fellow in The Institut f\"ur Experimentelle Mathematik,
Universit\"at Duisburg-Essen.

\section{Valuation-theoretic lemmas}

Let $A$ be a ring. By a {\bf valuation} of $A$ we shall mean a map $v: x\mapsto v(x)$ of $A$ onto a totally ordered commutative group $\Gamma$ (written addtively), together with an extra element $\infty$, such that:
\begin{enumerate}
\item $\alpha+\infty=\infty$ and $\alpha<\infty$ for all $\alpha\in \Gamma$.
\item $v(x)=\infty$ if and only if $x=0$.
\item $v(xy)=v(x)+v(y)$ for all $x,y\in A$.
\item $v(x+y)\geq \min\{v(x),v(y)\}$.
\end{enumerate}

If $A$ is a ring with a valuation $v$ on $A$, we shall also say simply that $A$ is a {\bf valued ring}.  The group $\Gamma$ is called the {\bf value group}.

\begin{lem} 
\label{lem:val1}
Let $\Gamma$ be a nontrivial totally ordered commutative group 
\begin{enumerate}
\item For any element $\gamma$ in $\Gamma$, there exists $\beta\in \Gamma$ such that $\beta< \gamma$.
\item Let $\gamma_1,\ldots,\gamma_r$ be elements in $\Gamma$ and let $n_1,\ldots,n_r$ be positive numbers. Then there exists an element $\gamma_0$ in $\Gamma$  such that for all elements $\gamma<\gamma_0$, $\gamma\in \Gamma$, we have $n_i\gamma<\gamma_i$ for all $i$.
\end{enumerate}
\end{lem}
\begin{proof}
1) If $\gamma \geq 0$, then let $\beta <0 \leq \gamma$ (such an element exists since $\Gamma$ is nontrivial). 

If $\gamma<0$, one can takes $\beta=2\gamma<\gamma$.  \\
2) We set 
$$\gamma_0:=\min\{\gamma_1,\ldots,\gamma_r, 0\}.$$ 
Now let $\gamma$ be an arbitrary element such that $\gamma<\gamma_0$.  Since $\gamma<\gamma_i$, $\gamma<0$, it follows that 
$n_i\gamma<\gamma_i$, for all $i$.
\end{proof}

\begin{lem}
\label{lem:val2}
Let $A$ be a valued ring of characteristic $p>0$ with nontrivial value group $\Gamma$. Let $f(T)=b_0T+\cdots+b_mT^{p^m}$ be a $p$-polynomial in one variable with coefficients in $A$. Then there exists an element $\gamma_0\in \Gamma$ such that if $a=f(a_1), a_1\in A$ and $v(a)<\gamma_0$ then $v(a)=v(b_m)+p^mv(a_1)$.
\end{lem}

\begin{proof}
By Lemma~\ref{lem:val1}, there exists an element $\alpha \in \Gamma$ such that  for all $\gamma<\alpha$ in $\Gamma$, we have
$$(p^m-p^i)\gamma< v(b_i)-v(b_m), \;\,\forall \;0\leq i < m.$$
We set 
$$\beta:=\min \{v(b_i)+\alpha p^i \,|\, 0\leq i\leq m\}.$$
Let $\gamma_0$ be any element with $\gamma_0<\beta$. Now assume that $a=f(a_1)$ such that $v(a)<\gamma_0$ ($a_1\in A$). Let $s$ be an index such that 
$$v(b_sa_1^{p^s}) =\min \{v(b_ia_1^{p^i}) \,|\, 0\leq i\leq m\}).$$
Then 
$$v(b_s)+p^s \alpha>\gamma_0>v(f(a_1))\geq v(b_s)+p^sv(a_1).$$
Thus $0<p^s(\alpha-v(a_1))$ and hence $v(a_1)<\alpha$. By the choice of $\alpha$, we have
$$v(b_ia_1^{p^i})=v(b_i)+p^iv(a_1)>v(b_m)+p^mv(a_1)=v(b_ma_1^{p^m}), \forall i<m.$$
Therefore $v(a)= v(b_m)+p^mv(a_1)$ as required.
\end{proof}

\begin{lem}
\label{lem:val3}
Let $\Gamma$ be a non-$p$-divisible totally ordered commutative group. Let $\alpha_0,\gamma_0$ be elements in $\Gamma$. Then there exist infinitely many elements $\gamma_i\in \Gamma$ such that 
$$\gamma_0>\gamma_1>\cdots >\gamma_i>\cdots$$  
and $\gamma_i\not\in \alpha_0+ p\Gamma,$ for all $i>0$.
\end{lem}

\begin{proof}
We first consider the case  $\alpha_0=0$. Since $\Gamma$ is not $p$-divisible, there is an element $a_0 \in \Gamma$ such that $a_0\not\in p\Gamma$. By Lemma~\ref{lem:val1} part (2), there exists an element $\delta_0$ such that for all $\delta<\delta_0$, we have $p\delta<\gamma_0=a_0$. By Lemma~\ref{lem:val1} part (1), there exists an infinite sequence
$$\delta_0>\delta_1>\cdots>\delta_i>\cdots.$$
Set $\gamma_i:=a_0+p\delta_i$, for all $i>0$. Then $\gamma_i\not\in p\Gamma,$ for all $i>0$ and $\gamma_0>\gamma_1>\cdots >\gamma_i>\cdots$. 

For the general case, applying the previous argument for $\gamma_0^\prime:=\gamma_0-\alpha_0$, we get an infinite sequence $\gamma_0^\prime>\gamma_1^\prime>\cdots >\gamma_i^\prime>\cdots$ with $\gamma_i^\prime \in \Gamma$ but $\gamma_i^\prime\not\in p\Gamma$. Setting $\gamma_i:=\gamma_i^\prime+\alpha_0$, we get a desired sequence of elements.

\end{proof}

\begin{prop}
\label{prop:infinite}
Let $A$ be a valued ring of characteristic $p>0$ with non-$p$-divisible value group $\Gamma$. Let $f(T)=b_0T+b_1T^p+\cdots+b_mT^{p^m}$ be a $p$-polynomial in one variable with coefficients in $A$ with $m\geq 1$ and $b_m\not=0$. Then $A/f(A)$ is infinite.
\end{prop}

\begin{proof} 
Let $\gamma_0$  be as in Lemma~\ref{lem:val2}. For any $a$ in $A$ such that $v(a) < \gamma_0$  and  $v(a)\not\in v(b_m)+p\Gamma$, Lemma~\ref{lem:val2} implies  that $a$ is not in $f(A)$. By Lemma~\ref{lem:val3} and noting that the valuation map $v$ is onto, we may choose a
sequence $\{a_i\}$ of elements from $A$ such that $v(a_i) \not \in v(b_m)+p\Gamma$ for all $i$, and  $\gamma_0 > v(a_1) > v(a_2)
> \cdots > v(a_i)>\cdots $. For every $i<j$, one has 
$$v(a_i-a_j)=v(a_i)\not\in v(b_m)+p\Gamma,$$ so $a_i-a_j\not\in f(A)$ and hence $a_i,a_j$ have different images in $A/f(A)$. Therefore, $A/f(A)$ is infinite. 
\end{proof}

\section{Proof of Theorem~\ref{thm:main} and a corollary}
\label{sec:3}
An {\bf embedding problem} $\sE$ for a profinite group $\Pi$ is a diagram 
\[
\xymatrix
{
{} & \Pi \ar[d]^{\alpha}\\
\Gamma \ar[r]^f & G
}
\] 
which consists of a pair of profinite groups $\Gamma$ and $G$ and epimomorphisms $\alpha:\Pi\to G$, $f:\Gamma\to G$. 

A {\bf weak solution} of $\sE$ is  a homomorphism $\beta:\Pi\to \Gamma$ such that $f\beta=\alpha$. If such a $\beta$ is surjective, then it is called a {\bf proper solution}. We will call $\sE$  {\bf weakly} (resp. {\bf properly}) {\bf solvable} if it has a weak (resp. proper) solution. 

We call $\sE$ a {\bf finite} embedding problem if the group $\Gamma$ is finite.

The {\bf kernel} of $\sE$ is defined to be $N:=\ker(f)$. We call $\sE$ a {\bf p-embedding problem} 
if $N$ is a  $p$-group. 

We say $\sE$ is a {\bf split} embedding problem if $f:\Gamma\to G$ has a group theoretical section, i.e., $f'\colon G\to \Gamma$ such that $f  f'$ is the identity map on $G$.

In this note, by a {\bf $K$-group}, where $K$ is a field, we mean  an algebraic affine group scheme which is smooth (\cite{Wa}). This notion is equivalent to the notion of a linear algebraic group defined over $K$ in the sense of \cite{Bo}.

First we need the following lemma.
\begin{lem}
\label{lem:additivepol}
Let $K$ be an infinite field of characteristic $p>0$. Let $P$ be a nontrivial finite commutative $K$-group which is  annihilated by $p$. Then $G$ is $K$-isomorphic to a $K$-subgroup of the additive group $\G_a$, of the form $\{x \mid f(x)=0\},$
where $f(T)=T+b_1T^p+\cdots+b_mT^{p^m}$ is a $p$-polynomial with coefficients in $K$, $m\geq 1$ and $b_m\not=0$. 
\end{lem}

\begin{proof} This is well known, see e.g.\ \cite[Proposition B.1.13]{CGP} or \ \cite[Chapitre V, Proposition 4.1 and Subsection 6.1]{Oe}.
\end{proof}

We are now ready to prove  Theorem~\ref{thm:main}.

\begin{proof}[Proof of Theorem~\ref{thm:main}]  
We have $\cd_p(\Gal(K))\leq 1$ (see, e.g., \cite[Chapter II, Proposition 3]{Se2}). By  Theorem~\ref{thm:Harbater} it suffices to prove that  $\Gal(K)$ is strongly $p$-dominating. 

Indeed, let $P$ be a non-trivial elementary $p$-group on which $\Gal(K)$ acts. 
Consider $P$ as a finite $K$-group. Then $P$ is commutative and annihilated by $p$. Hence by Lemma~\ref{lem:additivepol}, $P$ is $K$-isomorphic to a subgroup of $\G_a$ defined as the kernel of $f: \G_a\to \G_a$,  where $f(T)=T+\cdots+b_mT^{p^m}$ is a $p$-polynomial in one variable with coefficients in $K$ with $m\geq 1$ and  $b_m\not=0$. We have the following exact sequence of $K$-groups
$$
0\to P\to \G_a \stackrel{f}{\to}\G_a\to 0.  
$$
From this exact sequence we get the following exact sequence of Galois cohomology groups
$$
H^0(K,\G_a)\stackrel{f}{\to} H^0(K,\G_a)\to H^1(K,P)\to H^1(K,\G_a).
$$

By Hilbert 90 $H^1 (K,\G_a)=0$ (see e.g.\ \cite[Chapter II,  Proposition 1]{Se2}), hence
$$
H^1(K,P)\simeq H^0(K,\G_a)/\im(f)=K/f(K).
$$
The latter is infinite by Proposition~\ref{prop:infinite}. So we conclude that $H^1(K, P)$ is infinite, and hence $\Gal(K)$ is strongly $p$-dominating. 
\end{proof}

We recall that a Hilbertian field is a field $K$ having the irreducible specialization property: for every irreducible polynomial $f(T,X)\in k[T,X]$ that is separable in $X$, there exists $a\in K$ such that $f(a,X)$ is irreducible in $k[X]$ (we refer readers to \cite[Chapters 12, 13]{FJ}  for more details about Hilbertian fields). In \cite{DD}, D\`ebes and Deschamps give the following conjecture.
\begin{conj}[{\cite[2.1.2]{DD}}] 
\label{conj:DD}
Let $K$ be a Hilbertian field. Then every finite split embedding problem for $\Gal(K)$ has a proper solution.
\end{conj}

An easy consequence of Theorem~\ref{thm:main} is a simple proof of \cite[Theorem 8.3]{MM} which asserts that Conjecture~\ref{conj:DD} holds true whenever $K$ is of characteristic $p>0$ and if the kernel of the embedding problem is a $p$-group. Namely, we have

\begin{cor} 
\label{cor:DD}
Let $K$ be a Hilbertian field of characteristic $p>0$. Then every finite $p$-embedding problem for $\Gal(K)$ is  properly solvable.
\end{cor}

\begin{proof}
Let $\sE= (\alpha: \Gal(K)\to A, f: B\to A)$ be a finite $p$-embedding problem for $\Gal(K)$. Consider the finite $p$-embedding problem $\sE_t:=(\alpha\circ pr_t:\Gal(K(t))\to A, f:B\to A )$ for $\Gal(K(t))$ obtained by composition with the restriction map $\Gal(K(t))\to \Gal(K)$. Since $K(t)$ has discrete valuations,  Theorem~\ref{thm:main} gives a proper solution of $\sE_t$, say $\theta_t\colon \Gal(K(t))\to B$. By the irreducible specialization property (applied to a polynomial a root of which generates the solution field of $\theta_t$ over $K(t)$) $\theta_t$ specializes to a proper solution $\theta$ of $\sE$  (see \cite[Lemma 16.4.2]{FJ}).
\end{proof}

\begin{rmk} 
\begin{enumerate}
\item Let $G$  be a finite $p$-group, $K$ a Hilbertian field of characteristic $p>0$. By considering the finite (split) $p$-embedding problem $(\Gal(K)\to \{1\},G\to \{1\})$, Corollary~\ref{cor:DD} implies that $G$ is realizable over $K$. In other words, this proposition shows that every finite $p$-group is realizable over an arbitrary Hilbertian field of characteristic $p>0$. This last statement is a special case of a theorem of Shafarevich, \cite[Theorem 16.4.7]{FJ}.

\item Corollary~\ref{cor:DD} can also be derived from Ikeda's theorem (\cite[Proposition 16.4.5]{FJ}).  Here we sketch the proof: one starts with a  finite embedding problem for $K$ corresponding to an exact squence $1\to P \to B \to A\to 1$, where $P$ is a $p$-group and $B = \Gal(L/K)$. We use the usual trick of decomposing this embedding problem to a series of embedding problems in order to assume that $P$ is a minimal normal subgroup of $B$. In particular $P$ is \emph{abelian}. Since $\cd_p(K) \leq 1$ we can replace this embedding problem by a bigger \emph{split} embedding problem with the same kernel by taking the fiber product of $B$ and the image of a weak solution. Now we use Ikeda's result that gives a \emph{regular} solution over $K$, i.e., a solution  over $K(t)$ with the extra condition that the solution field is regular  over $L$. Then one uses Hilbertianity to reduce the solution to a solution over $K$. 

Unfortunately, we do not know whether any finite $p$-embedding problem over a field of characteristic $p>0$ has a regular solution.

\item For recent results concerning Conjecture~\ref{conj:DD}, we refer readers to \cite{BP,Pa1, Pa2, Po}.
\end{enumerate}
\end{rmk}

\section{Embedding problems with $p$-kernel}
In this section we show that the converse of Theorem~\ref{thm:main} also holds true, see Theorem~\ref{thm:converse}. 
 
 Let 
\[ \sE:= 
\xymatrix
{
{} & {} & {} & \Pi \ar[d]^{\alpha}\\
1 \ar[r] & P \ar[r] &\Gamma \ar[r]^f & G\ar[r] & 1
}
\]
be an embedding problem for $\Pi$ with abelian kernel $P$. Since $P$ is abelian, there is an induced conjugation action of $G$ on P by choosing representatives in $\Gamma$. This in turn yields an action of $\Pi$ on $P$ via $\alpha:\Pi\to G$. Let $H^1(\Pi,P)$ be the corresponding Galois cohomology group. 

Two weak solutions $\beta$ and $\beta^\prime: \Pi\to \Gamma$ of $\sE$ are defined to be equivalent, and denoted by $\beta \sim \beta^\prime $, if there is an element $p$ in $P$ such that $\beta^\prime= {\rm inn}(p)\circ \beta$. (Here ${\rm inn}(p)\in {\rm Aut}(\Gamma)$ denotes left conjugation by $p$.) One can check that $\sim$ is an equivalence on the set of weak solutions to $\sE$. Denote by ${\rm WS}(\sE)$ the set of weak solutions of $\sE$ modulo the equivalence relation $\sim$.  We have a cohomological description of ${\rm WS}(\sE)$. 

\begin{lem}
\label{lem:WS}
With notation as above, assume that $\sE$ is weakly solvable. Then 
${\rm WS}(\sE)$ is a $H^1(\Pi,P)$-torsor. 
In particular, any weak solution $\theta$ of $\sE$ induces a bijection 
\[
{\rm WS}(\sE) \cong H^1(\Pi,P).
\]
\end{lem}

\begin{proof} 
See \cite[Proposition 9.4.4]{NSW}.
\end{proof}

Next we prove the converse of Theorem~\ref{thm:Harbater}. For future reference we formulate it as an if and only if theorem. 

\begin{thm}
\label{thm:converse}
 Let $\Pi$ be  a  profinite group.  Then every finite $p$-embedding problem for $\Pi$ has a proper solution if and only if $\cd_p(\Pi)\leq 1$ and $\Pi$ is strongly $p$-dominating.
\end{thm}
\begin{proof}
$(\Leftarrow)$: This is Theorem~\ref{thm:Harbater}. 

$(\Rightarrow)$: It suffices to prove that $\Pi$ is strongly $p$-dominating. Let $P$ be a nontrivial elementary abelian $p$-group on which $\Pi$ acts continuously. We have to show that $H^1(\Pi,P)$ is infinite.

Since the action of $\Pi$ on $P$ is continuous, it factors via a finite quotient. I.e., there is a  map $\alpha\colon \Pi\to G$ and an action of $G$ on $P$ that induces the action of $\Pi$ on $P$.  Let $\Gamma$ be the semidirect product of $P$ and $G$. We get the following split embedding problem with elementary abelian $p$-kernel
\[ \sE:= 
\xymatrix
{
{} & {} & {} & \Pi \ar[d]^{\alpha}\\
1 \ar[r] & P \ar[r] &\Gamma \ar[r]^f & G\ar[r] & 1
}
\]
For any $n>0$ let $\Gamma_G^n$ be the $n$-th fold fiber product of $\Gamma$ over $G$, i.e., 
$$\Gamma_G^n=\{(\gamma_1,\ldots, \gamma_n), \gamma_i\in \Gamma,\ \textup{and } f(\gamma_1)=\cdots=f(\gamma_n)\in G\}.$$ 
We have a map $f_n\colon \Gamma_G^n\to G$, defined by $f_n((\gamma_i)_{i=1}^n) = f(\gamma_1)$. 

We have an embedding problem $\sE_n$ for $\Pi$ corresponding to the exact sequence
$$\xymatrix@1{
1\ar[r]& P^n\ar[r]& \Gamma_G^n\ar[r]^{f_n}& G \ar[r]& 1.
} $$ 
By assumption, there is a proper solution $\beta$ to $\sE_n$. By composing $\beta$ with the projections ${\rm pr}_i: \Gamma_G^n \to \Gamma$, we get $n$ proper solutions $\beta_1,\ldots, \beta_n$. 

We show that these $\beta_i$ are pairwise non-equivalent (and in particular distinct). Indeed, if $\beta_i\sim \beta_j$, for some  $1\leq i<j\leq n$, then there is a element $p\in P$ such that $\beta_i(s)=p  \beta_j(s) p^{-1}$, for all $s \in \Pi$. Since $P$ is a non-trivial group, we can take two different elements $q,q^\prime$ from $P$. Set $x=(1,\ldots,q,\ldots,q^\prime,\ldots,1)\in  \Gamma^n$, where $q,q^{\prime}$ are in $i$-th and $j$-th entry, respectively and  $1$ is in all other entries. Then $x\in\Gamma_G^n$. Since $\beta$ is a \emph{proper} solution, there exists $s$ in $\Pi$ such that $\beta(s)=x$. We then have
$$q=\beta_i(s)=p\beta_j(s)p^{-1}=pq^\prime p^{-1}=q^\prime,$$
a contradiction.

Therefore, we get that ${\rm WS}(\sE)$ is infinite, and by  Lemma~\ref{lem:WS}, $H^1(\Pi,P)$ is infinite, as needed.  
\end{proof}

\end{document}